\documentclass[a4paper]{article}

\usepackage{amsmath,amssymb}
\usepackage{amsfonts}
\usepackage{latexsym}
\usepackage{amsthm} 

\theoremstyle{plain}
\newtheorem{thm}{Theorem}[section]
\newtheorem{prop}[thm]{Proposition}
\newtheorem{lem}[thm]{Lemma}
\newtheorem{cor}[thm]{Corollary}

\theoremstyle{definition}
\newtheorem{con}[thm]{Condition}

\theoremstyle{remark}

\makeatletter
\@addtoreset{equation}{section}

\makeatother

\title{Rellich's theorem for spherically symmetric repulsive Hamiltonians}
\author{K. Itakura}
\date{}

\begin{document}

\maketitle

\begin{abstract}

For spherically symmetric repulsive Hamiltonians
we prove Rellich's theorem, or identify the largest weighted space of Agmon-H\"ormander type
where the generalized eigenfunctions are absent.
The proof is intensively dependent on commutator arguments.
Our novelty here is a use of conjugate operator associated with some radial flow,
not with dilations and translations.
Our method is simple and elementary, and does not employ any advanced tools such as
the operational calculus or the Fourier analysis.

\end{abstract}


\section{Introduction}
For any fixed $\epsilon\in (0,2]$
we consider the repulsive Schr\"{o}dinger operator
$$
H = -\frac{1}{2}\Delta - |x|^{\epsilon} +q; \quad -\Delta = p_j\delta^{jk}p_k, \ p_j = -i\partial_{x_j},
$$
on the Hilbert space ${\mathcal H} = L^2({\mathbb R}^d)$. 
Here $\delta^{jk}$ is the Kronecker delta, and we use the Einstein summation convention (throughout the paper we will use this notation),
and $q$ is a real-valued function that may grow slightly slower than $|x|^\epsilon$. 
By the Faris-Lavine theorem (see \cite[I\hspace{-1pt}I]{rs}) 
the operator $H$ is essentially self-adjoint on $C^\infty_0(\mathbb R^d)$,
and we denote the self-adjoint extension by the same letter.

The purpose of this paper is to prove Rellich's theorem for $H$,
which asserts the absence of generalized eigenfunctions 
in a certain weighted space, the Agmon-H\"{o}rmander space.
The space is optimal in the sense that we can actually construct a generalized eigenfunction in any larger spaces.
For the proof we apply a new commutator argument with some {\it weight inside} invented recently by \cite{is}.
A feature of this argument is a choice of the conjugate operator $A$:
We choose $A$, in Section 2, to be a generator of some radial flow, not of dilations and translations.
The proof consists only of direct computations and estimations of commutators, 
and does not require any deep knowledge from functional analysis or Fourier analysis.

\subsection{Setting}
Choose $\chi \in C^{\infty}({\mathbb R})$ such that
\begin{equation}\label{chi}
\chi (t) =
\begin{cases}
1 & {\rm for} \ t \leq 1, \\
0 & {\rm for} \ t \geq 2,
\end{cases}
\quad \chi' \leq 0,
\end{equation}
and set $r \in C^{\infty}({\mathbb R}^d)$ as
\begin{equation}\label{r}
r(x) = \chi(|x|)+|x|\left( 1-\chi(|x|) \right).
\end{equation}

\begin{con} \label{con}
The perturbation $q$ is a real-valued function.
Moreover, there exists a splitting by real-valued functions:
$$q_0 = q_1 + q_2 ; \quad q_1 \in C^1({\mathbb R}^d), \ q_2 \in L^{\infty}({\mathbb R}^d),$$
such that for some $\rho ,C>0$ the following bounds hold globally on ${\mathbb R}^d$:
\begin{equation*}
|q_1| \leq Cr^{\epsilon-\rho}, \quad \nabla^rq_1 \leq Cr^{\epsilon-1-\rho}, \quad |q_2|\leq Cr^{\epsilon/2-1-\rho}.
\end{equation*}
\end{con}

\vspace{2mm}
We introduce the weighted Hilbert space ${\mathcal H}_s$ for $s \in {\mathbb R}$ by
$${\mathcal H}_s = r^{-s}{\mathcal H}.$$
We also denote the locally $L^2$-space by
$${\mathcal H}_{\rm loc} = L^2_{\rm loc}({\mathbb R}^d).$$
We consider $B_R = \{ r<R \}$ and the characteristic functions
\begin{equation*}
F_{\nu} =
F(B_{R_{\nu+1}}\setminus B_{R_{\nu}}),\quad R_{\nu} = 2^{\nu}, \ \nu \geq 0,
\end{equation*}
where $F(\Omega)$ denotes sharp characteristic function of a subset $\Omega \subseteq {\mathbb R}^d$.
Define the spaces ${\mathcal B}^*$and ${\mathcal B}^*_0$ by
\begin{equation*}
\begin{split}
{\mathcal B}^* &= \{ \psi \in {\mathcal H}_{\rm loc} \ | \ \|\psi\|_{{\mathcal B}^*} < \infty \}, \quad \ \|\psi\|_{{\mathcal B}^*} = \sup_{\nu \geq 0} R_{\nu}^{\epsilon/4-1/2}\| F_{\nu}\psi\|_{\mathcal H},
\\
\mathcal B^*_0&=\{\psi\in \mathcal B^*\ |\ \lim_{\nu \to \infty} R_{\nu}^{\epsilon/4-1/2}\| F_{\nu}\psi\|_{\mathcal H}
=0\},
\end{split}
\end{equation*}
respectively. 
We note that ${\mathcal B}_0^*$ coincide with the closure of $C_0^{\infty}({\mathbb R}^d)$ in ${\mathcal B}^*$.
If $0<\epsilon<2$, the following inclusion relations hold for any $s>1/2-\epsilon/4$:
\begin{align}\label{1720221503}
{\mathcal H} \subsetneq {\mathcal H}_{\epsilon/4-1/2} \subsetneq {\mathcal B}_0^* \subsetneq {\mathcal B}^* \subsetneq {\mathcal H}_{-s}.
\end{align}
Similarly, if $\epsilon=2$, the following inclusion relations hold for any $s>0$:
\begin{align}\label{1720221504}
{\mathcal H} \subsetneq {\mathcal B}_0^* \subsetneq {\mathcal B}^* \subsetneq {\mathcal H}_{-s}.
\end{align}

\subsection{Results}
Our main result is the absence of ${\mathcal B}_0^*$-eigenfunctions for any eigenvalues $\lambda \in {\mathbb R}$.

\begin{thm}\label{1702221514}
Suppose Condition \ref{con}, and let $\lambda \in {\mathbb R}$.
If a function $\phi \in \mathcal B_0^*$ satisfies that
$$
(H-\lambda)\phi =0,
$$
in the distributional sense, then $\phi =0$ in ${\mathbb R}^d$.
\end{thm}

By the inclusions \eqref{1720221503} and \eqref{1720221504}
we obviously have the following corollary.

\begin{cor}
The operator $H$ has no eigenvalues: $\sigma_{{\rm pp}}(H) = \emptyset$.
\end{cor}

As we will see in Subsection 1.3 below,
we can actually construct a ${\mathcal B}^*$-eigenfunction, 
and hence the function space $\mathcal B_0^*$ in Theorem~\ref{1702221514} is optimal.
Note that Theorem~\ref{1702221514} covers the one-dimensional Stark Hamiltonians,
and in this case the Airy function exactly provides a ${\mathcal B}^*$-eigenfunction.
As far as we know, there seem to be no results on Rellich's theorem for repulsive Hamiltonians so far, and our result is new.
To prove Theorem~\ref{1702221514} we apply a new commutator argument with some {\it weight inside}
from \cite{is}.
We are directly motivated by their result,
in which spectral properties of the Schr\"odinger operator on manifold with ends are studied.
However, they consider only potentials decaying at infinity.
In order to deal with repulsive potentials that diverge to $-\infty$ at infinity
we need to appropriately change a construction of the conjugate operator (see \eqref{A0}).

In case $\epsilon=0$, there has been an extensive amount of literature on 
eigenvalue problems (e.g. \cite{a, fh, fhh2o, ho, ij, is, iso}).
As for the case $\epsilon=2$, Ishida studied inverse scattering problem in \cite{ishi}.
Our setting excludes $\epsilon>2$, however, 
Matsumoto, Kakazu and Nagamine studied eigenvalue problems for $\epsilon>2$ in \cite{mkn}.
Skibsted studied the case where $H$ has an attractive potential in \cite{s},
whereas we considered the case where $H$ has a repulsive potential.
Skibsted showed Rellich's theorem, in \cite{s}, as a corollary of a uniqueness theorem of the outgoing solution at zero energy.
We also mention a recent result \cite{im} by Isozaki and Morioka 
that studies Rellich's theorem 
for discrete Schr\"{o}dinger operator.

In Subsection 1.3 below, we verify existence of a generalized eigenfunction in ${\mathcal B}^*$.
In Section 2, we introduce the conjugate operator $A$ and show that $A$ is the generator of a strongly continuous one-parameter unitary group of
some radial flow.
In addition, we introduce commutators with weight inside and discuss the properties.
In Section 3, we prove Theorem~\ref{1702221514}.
The proof consists of two ingredients that are typical in such a topic:
a priori super-exponential decay estimate and 
the absence of super-exponentially decaying eigenfunctions.
In the proofs of the both statements commutator estimates play important rolls.

\subsection{Existence of ${\mathcal B}^*$-eigenfunctions}

In this subsection, we show optimality of Theorem~\ref{1702221514}.
To show that we construct a spherically symmetric solution $\psi(x)=\tilde\psi(\tau) \in {\mathcal B}^*$ of
\begin{equation} \label{b*ef}
\left( \frac{1}{2}\Delta + |x|^{\epsilon} \right)\psi=\left( \frac{1}{2}\frac{d^2}{d\tau^2}+\frac{d-1}{2\tau}\frac{d}{d\tau}+\tau^{\epsilon} \right)\tilde\psi=0,
\end{equation}
where $\tau=|x|, \ x \in {\mathbb R}^d$.

Recall the Bessel equation:
\begin{equation} \label{besseleq}
s^2f''(s)+sf'(s)+\left( s^2-\nu^2 \right)f(s)=0,
\end{equation}
and the Bessel function $J_{\nu}(s)$ which is one of the solutions of  \eqref{besseleq}.
In \eqref{besseleq}, we let $\nu=\frac{d-2}{\epsilon+2}$, and change variables by
$$
s=\frac{2\sqrt{2}}{\epsilon+2}\tau^{\epsilon/2+1}, \quad f=\tau^{d/2-1}\tilde\psi.
$$
Then we obtain the following equation:
\begin{equation} \label{psieq}
\left( \frac{1}{2}\frac{d^2}{d\tau^2}+\frac{d-1}{2\tau}\frac{d}{d\tau}+\tau^{\epsilon} \right)\tilde\psi=0.
\end{equation}
Hence we obtain \eqref{b*ef}.
Now, we show $\psi \in {\mathcal B}^*$.
Since $\nu=\frac{d-2}{\epsilon+2}$, we can write by definition of the Bessel function
\begin{equation*}
J_{(d-2)/(\epsilon+2)}(s)=
\sum\limits_{m=0}^{\infty}\frac{(-1)^m}{m!\Gamma(m+(d-2)/(\epsilon+2)+1)}\left( \frac{s}{2} \right)^{2m+(d-2)/(\epsilon+2)}.
\end{equation*}
It is well-known that $J_{(d-2)/(\epsilon+2)}(s)=O(s^{-1/2})$ as $s \to \infty$ (e.g. \cite{k}).
Hence we have
\begin{equation} \label{psiord}
\psi(x)=\tilde\psi(\tau) = \tau^{1-d/2}J_{(d-2)/(\epsilon+2)}\left( \frac{2\sqrt{2}}{\epsilon+2}\tau^{\epsilon/2+1} \right)=O(\tau^{-d/2-\epsilon/4+1/2}), \ \tau \to \infty.
\end{equation}
By the following expression:
\begin{align*}
\psi(x) &= |x|^{1-d/2}J_{(d-2)/(\epsilon+2)}\left( \frac{2\sqrt{2}}{\epsilon+2}|x|^{\epsilon/2+1} \right) \\
  &= \sum_{m=0}^{\infty}\frac{(-1)^m}{m!\Gamma(m+(d-2)/(\epsilon+2)+1)}\left( \frac{\sqrt{2}}{\epsilon+2} \right)^{2m+(d-2)/(\epsilon+2)}|x|^{(\epsilon+2)m},
\end{align*}
we have
\begin{equation} \label{psibibun}
\psi \in C^2(\mathbb R^d).
\end{equation}
\eqref{psiord} and \eqref{psibibun} imply $\psi \in {\mathcal B}^*$.
Therefore $\psi$ is a ${\mathcal B}^*$-eigenfunction for $H=-\frac{1}{2}\Delta-|x|^{\epsilon}$.


\section{Preliminaries}

In this section we prepare some tools to prove Theorem~\ref{1702221514}.
From this section, we use a geometric notation.
However, we consider the only case of Euclidean space.
Hence it suffices to note the following properties: for $g \in C^{\infty}(\mathbb R^d)$ and $i,j,k,l=1,2, \ldots, d$
\begin{align*}
(\nabla g)_i &= (\nabla_i g) = (d g)_i = \partial_i g, \quad (\nabla g)^i = \delta^{ij}(\nabla g)_j, \\
(\nabla^2 g)_{ij} &= \partial_i \partial_j g, \quad (\nabla^2 g)^{ij} = \delta^{ik}(\nabla^2 g)_k{}^j = \delta^{ik}\delta^{jl}(\nabla^2 g)_{kl}.
\end{align*}
We also remark that the following inequality holds:
\begin{equation*}
(\nabla g)^j(\nabla g)^k \leq |dg|^2\delta^{jk},
\end{equation*}
as quadratic form estimates on fibers of the tangent bundle of $\mathbb R^d$, i.e. for any $\xi \in {\mathbb R}^d$
$$
(\nabla g)^j(\nabla g)^k\xi_j\xi_k \leq |dg|^2\xi_j\delta^{jk}\xi_k = |dg|^2|\xi|^2.
$$

First, using the function $r \in C^{\infty}({\mathbb R}^d)$ of \eqref{r}, 
we introduce the conjugate operator $A$ as a maximal differential operator
\begin{equation} \label{A0}
A={\rm Re}\left( r^{-\epsilon/2}p^r \right), \quad p^r=-i\nabla^r, \ \nabla^r=(\nabla r)_j\delta^{jk}\nabla_k,
\end{equation}
with domain 
\begin{align*}
\mathcal D(A)=\{\psi\in\mathcal H\ |\ A\psi\in\mathcal H\}.
\end{align*}
Here, for notational simplicity, we set the function $f \in C^{\infty}([1, \infty))$ as
\begin{equation}\label{f}
f(r) =
\begin{cases}
\frac{1}{1-\epsilon/2}r^{1- \epsilon/2} & {\rm for} \ 0<\epsilon <2, \\
\log{r} & {\rm for} \ \epsilon = 2.
\end{cases}
\end{equation}
Then note that the conjugate operator $A$ has the following expressions:
\begin{equation} \label{A}
A = (p^r)^*r^{-\epsilon/2} + \frac{i}{2}(\Delta f) = r^{-\epsilon/2}p^r - \frac{i}{2}(\Delta f).
\end{equation}

\subsection{Unitary group and generator}

Let
$$
y:{\mathbb R} \times {\mathbb R}^d \to {\mathbb R}^d, \ (t, x) \mapsto y(t, x)=\exp(t\nabla f)(x),
$$
be the maximal flow generated by the gradient vector field $\nabla f$.
Note that by definition it satisfies
\begin{equation*}
\partial_t y^i(t, x) = (\nabla f)^i(y(t,x)), \quad y(0, x)=x.
\end{equation*}
We define $T(t): {\mathcal H} \to {\mathcal H},\, t \in {\mathbb R}$, by
\begin{equation} \label{defTt}
\begin{split}
(T(t)\psi)(x) &= J(t, x)^{1/2}\psi(y(t, x)) \\
  &= \exp\left( \int_0^t \frac{1}{2}(\Delta f)(y(t, x)) ds \right)\psi(y(t, x)),
\end{split}
\end{equation}
where $J(t, \cdot)$ is the Jacobian of the mapping $y(t, \cdot): {\mathbb R}^d \to {\mathbb R}^d$.
We can easily verify the equivalence of the two expressions in \eqref{defTt} by the following identity:
\begin{align*}
\partial_t \bigl[ \log J(t, x)^2 \bigr] &= 2J(t, x)^{-1}\partial_t J(t, x) \\
  &= 2J(t, x)^{-1}(\Delta f)(y(t, x))J(t,x) = 2(\Delta f)(y(t, x)).
\end{align*}

Now it follows by the upper expression of \eqref{defTt} that for any $\psi \in {\mathcal H}$
$$
\|T(t)\psi\|^2 = \int_{\mathbb R} |\psi(y(t, x))|^2J(t, x) dx = \|\psi\|^2.
$$
and hence $T(t)$, $t \in {\mathbb R}$, forms a strongly continuous one-parameter unitary group.

Next we investigate the generator $A_0$ of group $T(t),\, t \in {\mathbb R}$.
By definition 
\begin{align*}
{\mathcal D}(A_0) &= \{ \psi \in {\mathcal H} \, |\, \lim_{t \to 0}(it)^{-1}(T(t)\psi - \psi) \ {\rm exists \ in}\ {\mathcal H}\}, \\
A_0\psi &= \lim_{t \to 0}(it)^{-1}(T(t)\psi - \psi) \quad {\rm for} \ \psi \in {\mathcal D}(A_0).
\end{align*}
By the Stone theorem the generator $A_0$ is self-adjoint on ${\mathcal H}$.
It is easy to verify that $C_0^\infty(\mathbb R^d)\subset \mathcal D(A_0)$, and that 
$T(t)$ preserves $C_0^{\infty}({\mathbb R}^d)$.
Hence by \cite[Theorem X.49]{rs} the space $C_0^{\infty}({\mathbb R}^d)$ is a core for $A_0$.
It is also clear by definition that on $C_0^{\infty}({\mathbb R}^d)$ the generator $A_0$ and 
maximal differential operator $A$ coincides,
and therefore they are actually the identical operators:
\begin{equation*}
{\mathcal D}(A_0)={\mathcal D}(A), \quad A_0=A={\rm Re}\left( r^{-\epsilon/2}p^r \right)=\frac{1}{2}\left( r^{-\epsilon/2}p^r + (p^r)^*r^{-\epsilon/2} \right).
\end{equation*}

\begin{lem}
Let $H^2(\mathbb R^d)$ be the Sobolev space of second order, and set
$$
H^2_{\rm comp}(\mathbb R^d) = \{ \psi \in H^2(\mathbb R^d) \,|\, {\rm supp}\,\psi \ {\rm is \ compact}\}.
$$
Then the following inclusion relations hold.
\begin{equation} \label{embed}
H^2_{\mathrm{comp}}(\mathbb R^d) \subset {\mathcal D}(H) \subset {\mathcal D}(A).
\end{equation}
\end{lem}
\begin{proof}
First we prove $H^2_{\mathrm{comp}}(\mathbb R^d) \subset {\mathcal D}(H)$.
Let $\psi \in H^2_{\rm comp}(\mathbb R^d)$ and ${\rm supp}\,\psi =K$, and set
$$
K_1=\left\{x \in {\mathbb R}^d\ | \ \inf_{y \in K}|x-y| \leq 1 \right\}.
$$
Then there exists $\{\psi_n\} \subset C_0^{\infty}(\mathbb R^d)$ such that
\begin{equation} \label{psin}
{\rm supp}\,\psi_n \subset K_1, \quad \| \psi_n - \psi \|+\| p^2(\psi_n - \psi)\| \to 0 \ \ {\rm as} \ n \to \infty.
\end{equation}
By \eqref{psin} and Condition~\ref{con} we can estimate as follows.
\begin{equation*}
\|H\psi_n - H\psi\| + \|\psi_n - \psi\| \leq \frac{1}{2}\|p^2(\psi_n-\psi)\|+C_K\|\psi_n-\psi\| \to 0 \ \ {\rm as}\ \ n \to \infty.
\end{equation*}
This implies $\psi \in {\mathcal D}(H)$.

Now we prove ${\mathcal D}(H) \subset {\mathcal D}(A)$.
Let us discuss similarly to \cite{sig}.
Let $\psi \in {\mathcal D}(H)$.
It suffices to show that $r^{-\epsilon/2}\psi \in {\mathcal D}\left( p^r \right)={\mathcal D}\left( (\nabla r)^jp_j \right)$.
We choose $\eta \in C_0^{\infty}(\mathbb R^d)$ such that for any multi-index $\alpha$
$$
0 \leq \eta \leq 1, \quad |\partial^{\alpha}\eta|\leq C_{\alpha}, \ C_{\alpha}>0,
$$
and we let $\omega=\eta r^{-\epsilon/2}\psi$.
Then we have $\omega \in {\mathcal D}(p^2)$ and we obtain the following estimate:
\begin{equation} \label{emb1}
\| (\nabla r)^jp_j\omega \|^2 \leq C_1\| p\omega\|^2 = 2C_1\langle H \rangle_{\omega} + 2C_1\langle |x|^{\epsilon}-q \rangle_{\omega},
\end{equation}
where in general for a linear operator $T$ we write
$$
\langle T \rangle_{\omega} = \langle \omega, T\omega \rangle.
$$
We estimate the first term of \eqref{emb1} by
\begin{equation} \label{emb2}
\begin{split}
2C_1\langle \omega, H\omega \rangle &= 2C_1\langle \psi, \eta r^{-\epsilon}\eta H\psi \rangle + 2C_1{\rm Re}\langle \psi, \eta r^{-\epsilon/2}[H, \eta r^{-\epsilon/2}]\psi \rangle \\
  &\leq C_2\|\psi\| \|H\psi\| + \| |\nabla \eta r^{-\epsilon/2}|\psi \|^2 \\
  &\leq C_3\|\psi\|^2 + C_3\| H\psi \|^2 < \infty.
\end{split}
\end{equation}
Using the Condition~\ref{con} we can estimate the second term of \eqref{emb1} as
\begin{equation} \label{emb3}
\langle |x|^{\epsilon}-q \rangle_{\omega} \leq C_4\|\psi\|^2 < \infty.
\end{equation}
Hence we obtain by \eqref{emb1}, \eqref{emb2} and \eqref{emb3}
$$
\|(\nabla r)^jp_j\omega\| \leq C_5.
$$
It provides
$$
\|\eta (\nabla r)^jp_jr^{-\epsilon/2}\psi\| \leq C_6.
$$
Hence by Lebesgue's monotone convergence theorem we obtain
$$
\|(\nabla r)^jp_jr^{-\epsilon/2}\psi\| \leq C_7.
$$
We are done.
\end{proof}

\subsection{Commutators with weight inside}

Next we consider commutators with a weight $\Theta$ inside:
\begin{equation*}
[ H, iA ]_{\Theta} := i(H\Theta A - A\Theta H).
\end{equation*}
Let $\Theta =\Theta(r)$ be a non-negative smooth function with bounded derivatives.
More explicitly, if we denote its derivatives in $r$ by primes such as $\Theta'$, then
\begin{equation} \label{theta1}
\Theta \geq 0, \quad |\Theta^{(k)}|\leq C_k, \ \  k=0,1,2,\ldots .
\end{equation}
We first define the quadratic form $[H,iA]_{\Theta}$ on $C^\infty_0(\mathbb R^d)$,
and then extend it to $H^1(\mathbb R^d)$ when $\Theta$ is compactly supported according to 
the following lemma.

\begin{lem} \label{lem21}
Suppose Condition \ref{con}, and let $\Theta$ be a non-negative smooth function with bounded derivatives {\rm \eqref{theta1}}.
Then, as quadratic forms on $C_0^{\infty}({\mathbb R}^d)$,
\begin{align}
\begin{split} \label{commest0}
[H, iA]_{\Theta} &= p_j(\nabla^2 f)^{jk}\Theta p_k + (p^r)^*r^{-\epsilon/2}\Theta'p^r + \frac{1}{2}{\rm Re}\left( (\Delta f)\Theta p_i\delta^{ij}p_j \right) \\
  &\quad \, - \frac{1}{2}p_i(\Delta f)\Theta\delta^{ij}p_j - \frac{1}{2}{\rm Im}\left( (\nabla |dr|^2)^jr^{-\epsilon/2}\Theta'p_j \right) - {\rm Im}\left( 2q_2r^{-\epsilon/2}\Theta p^r \right) \\
  &\quad \, - {\rm Re}\left( |dr|^2r^{-\epsilon/2}\Theta' H \right) + \epsilon r^{-\epsilon/2}(\nabla r)^k|x|^{\epsilon-2}x_k\Theta + q_{\Theta} - \frac{1}{4}|dr|^4r^{-\epsilon/2}\Theta''';
\end{split} \\
q_{\Theta} &= - r^{-\epsilon/2}(\nabla^r q_1)\Theta + q_2(\Delta f)\Theta + \frac{\epsilon}{8}(\nabla^r |dr|^2)r^{-\epsilon/2-1}\Theta' \notag \\
  &\quad \, - \frac{\epsilon}{8}\left( \frac{\epsilon}{2}+1 \right)|dr|^4r^{-\epsilon/2-2}\Theta' + \frac{\epsilon}{8}|dr|^2r^{-\epsilon/2-1}(\Delta r)\Theta' + |dr|^2r^{-\epsilon/2}q_2\Theta' \notag \\
  &\quad \, + \frac{\epsilon}{4}|dr|^4r^{-\epsilon/2-1}\Theta'' - \frac{1}{4}(\nabla^r |dr|^2)r^{-\epsilon/2}\Theta'' - \frac{1}{4}|dr|^2r^{-\epsilon/2}(\Delta r)\Theta''. \notag
\end{align}
In particular, if $\Theta$ has a compact support, by the Cauchy-Schwarz inequality $[H, iA]_{\Theta}$ restricted to $C_0^{\infty}({\mathbb R}^d)$ extends to a bounded form on $H^1(\mathbb R^d)$.
\end{lem}


\begin{proof}
By \eqref{A} we obtain
\begin{align}
[H, iA]_{\Theta}
  &= 2{\rm Im}\left[A\Theta \left\{ \frac{1}{2}p_i\delta^{ij}p_j - |x|^{\epsilon} + q \right\} \right] \notag \\
  &= \frac{1}{2i}\left\{ p_ip_k(\nabla r)^kr^{-\epsilon/2}\Theta\delta^{ij}p_j - p_i\delta^{ij}\Theta r^{-\epsilon/2}(\nabla r)^kp_kp_j \right\} + p_j(\nabla^2 f)^{jk}\Theta p_k \notag \\
  &\quad + p_j(\nabla r)^jr^{-\epsilon/2}\Theta'(\nabla r)^kp_k + \frac{1}{2}{\rm Re}\left( (\Delta f)\Theta p_i\delta^{ij}p_j \right) + |dr|^2r^{-\epsilon/2}|x|^{\epsilon}\Theta' \notag \\
  &\quad + \epsilon r^{-\epsilon/2}(\nabla r)^k|x|^{\epsilon-2}x_k\Theta - |dr|^2r^{-\epsilon/2}q_1\Theta' - r^{-\epsilon/2}(\nabla^r q_1)\Theta - {\rm Im}\left( 2q_2\Theta A \right) \notag \\
\begin{split} \label{commest1}
  &= p_j(\nabla^2 f)^{jk}\Theta p_k + p_j(\nabla r)^jr^{-1}\Theta'(\nabla r)^kp_k + \frac{1}{2}{\rm Re}\left( (\Delta f)\Theta p_i\delta^{ij}p_j \right) \\
  &\quad - \frac{1}{2}p_i(\Delta f)\Theta\delta^{ij}p_j - \frac{1}{2}p_i|dr|^2r^{-\epsilon/2}\Theta'\delta^{ij}p_j + |dr|^2r^{-\epsilon/2}|x|^{\epsilon}\Theta' \\
  &\quad + \epsilon r^{-\epsilon/2}(\nabla r)_k|x|^{\epsilon-2}x^k\Theta - |dr|^2r^{-\epsilon/2}q_1\Theta' - r^{-\epsilon/2}(\nabla^r q_1)\Theta \\
  &\quad - {\rm Im}\left( 2q_2r^{-\epsilon/2}\Theta p^r \right) + q_2(\Delta f)\Theta.
\end{split}
\end{align}

\noindent
We combine the fifth, the sixth and the eighth terms of \eqref{commest1} as follows.
\begin{equation} \label{commest2}
\begin{split}
  &\quad \hspace{-3mm} - \frac{1}{2}p_i|dr|^2r^{-\epsilon/2}\Theta'\delta^{ij}p_j + |dr|^2r^{-\epsilon/2}|x|^{\epsilon}\Theta' - |dr|^2r^{-\epsilon/2}q_1\Theta' \\
  &= - \frac{1}{2}{\rm Im}\left( \left( \nabla |dr|^2 \right)^jr^{-\epsilon/2}\Theta'p_j \right) + \frac{\epsilon}{4}{\rm Im}\left( |dr|^2r^{-\epsilon/2-1}\Theta'p^r \right) \\
  &\quad - \frac{1}{2}{\rm Im}\left( |dr|^2r^{-\epsilon/2}\Theta''p^r \right) - {\rm Re}\left( |dr|^2r^{-\epsilon/2}\Theta' H \right) + |dr|^2r^{-\epsilon/2}q_2\Theta' \\
  &= - \frac{1}{2}{\rm Im}\left( \left( \nabla |dr|^2 \right)^jr^{-\epsilon/2}\Theta'p_j \right) - {\rm Re}\left( |dr|^2r^{-\epsilon/2}\Theta' H \right) + |dr|^2r^{-\epsilon/2}q_2\Theta' \\
  &\quad + \frac{\epsilon}{8}\left( \nabla^r |dr|^2 \right)r^{-\epsilon/2-1}\Theta' - \frac{\epsilon}{8}\left( \frac{\epsilon}{2}+1 \right)|dr|^4r^{-\epsilon/2-2}\Theta' + \frac{\epsilon}{8}|dr|^2r^{-\epsilon/2-1}(\Delta r)\Theta' \\
  &\quad + \frac{\epsilon}{4}|dr|^4r^{-\epsilon/2-1}\Theta'' - \frac{1}{4}\left( \nabla^r |dr|^2 \right)r^{-\epsilon/2}\Theta'' - \frac{1}{4}|dr|^2r^{-\epsilon/2}(\Delta r)\Theta'' \\
  &\quad - \frac{1}{4}|dr|^4r^{-\epsilon/2}\Theta'''.
\end{split}
\end{equation}
If we substitute \eqref{commest2} into \eqref{commest1}, then the expression \eqref{commest0} follows.

The boundedness of $[H, iA]_{\Theta}$ as a quadratic form on $H^1(\mathbb R^d)$
follows from \eqref{embed}, \eqref{commest0} and compactness of {\rm supp}$\,\Theta$.
\end{proof}

In the above argument we defined the weighted commutator $[H, iA]_{\Theta}$ as a quadratic form on 
$H^1(\mathbb R^d)$ as an extension from $C_0^{\infty}(\mathbb R^d)$.
On the other hand, throughout the paper, we shall use the notation
\begin{equation*} 
{\rm Im}(A\Theta H)=\frac{1}{2i}(A\Theta H - H\Theta A)
\end{equation*}
as a quadratic form defined on ${\mathcal D}(H)$, i.e. for $\psi \in {\mathcal D}(H)$
$$\langle {\rm Im}(A\Theta H) \rangle_{\psi} = \frac{1}{2i}\left( \langle A\psi, \Theta H\psi \rangle - \langle H\psi, \Theta A\psi \rangle \right).$$
Note that by the embedding \eqref{embed} the above quadratic form is well-defined.
Obviously the quadratic forms $[H, iA]_{\Theta}$ and $2{\rm Im}(A\Theta H)$ coincide on $C_0^{\infty}(\mathbb R^d)$,
and hence we obtain
\begin{equation}\label{equal}
[H, iA]_{\Theta} = 2{\rm Im}(A\Theta H) \quad {\rm on} \ {\mathcal D}(H),
\end{equation}
if $\Theta$ is compactly supported.
In fact, by the Faris-Lavine theorem for any $\psi \in {\mathcal D}(H)$ there exists $\{\psi_n\} \subset C_0^{\infty}({\mathbb R}^d)$ such that
$$
\|\psi-\psi_n\|+\|H(\psi-\psi_n)\| \to 0\quad\text{as }n\to\infty.
$$
Therefore we obtain
$$
\langle [H, iA]_{\Theta} \rangle_{\psi} = \lim_{n \to \infty}\langle [H, iA]_{\Theta} \rangle_{\psi_n} = \lim_{n \to \infty}\langle 2{\rm Im}(A\Theta H) \rangle_{\psi_n} = \langle 2{\rm Im}(A\Theta H) \rangle_{\psi}.
$$


\section{Proof of Theorem~\ref{1702221514}}

The proof of Theorem~\ref{1702221514} consists of two steps, 
a priori super-exponential decay estimates and the absence of super-exponentially decaying eigenfunctions.
Obviously, Theorem~\ref{1702221514} follows immediately as a combination of the following propositions.
Throughout the section we suppose Condition~\ref{con}.

\begin{prop} \label{prop}
Let $\lambda \in {\mathbb R}$. 
If a function $\phi \in {\mathcal B}_0^*$ satisfies that
$$
(H-\lambda)\phi =0,
$$
in the distributional sense, then $e^{\alpha r}\phi \in {\mathcal B}_0^*$ for any $\alpha \geq 0$.
\end{prop}

\begin{prop} \label{prop2}
Let $\lambda \in {\mathbb R}$. 
If a function $\phi \in {\mathcal B}_0^*$ satisfies that
\vspace{-0.7mm}
\begin{itemize} \setlength{\parskip}{-0.6mm}
  \item[(1)] $(H-\lambda)\phi =0$ in the distributional sense,
  \item[(2)] $e^{\alpha r}\phi \in {\mathcal B}_0^*$ for any $\alpha \geq 0$,
\end{itemize}
\vspace{-0.7mm}
then $\phi(x) = 0$ in ${\mathbb R}^d$.
\end{prop}

We prove Propositions~\ref{prop} and \ref{prop2} in Subsections 3.1 and 3.2, respectively.
The proofs are quite similar to each other, and both are dependent on commutator estimates with 
particular forms of weights inside.

Now, using the function $\chi \in C^{\infty}(\mathbb R)$ of \eqref{chi}, we define $\chi_n, \bar{\chi}_n, \chi_{m, n} \in C^{\infty}(\mathbb R)$ for  $n > m \geq 0$ by
\begin{equation*}
\chi_n = \chi (x/R_n), \quad \bar{\chi}_n = 1- \chi_n, \quad \chi_{m, n} = \bar{\chi}_m\chi_n,
\end {equation*}
and let us introduce the regularized weights
\begin{equation} \label{theta2}
\Theta = \Theta^{\alpha, \delta}_{m, n, \nu} = \chi_{m, n}e^{\theta}; \quad n> m \geq 0,
\end{equation}
with exponents
$$
\theta = \theta^{\alpha, \delta}_{\nu} = 2\alpha \int^r_0 (1+s/R_{\nu})^{-1-\delta} ds; \quad \alpha \geq 0, \  \delta>0, \  \nu \geq 0.
$$
Denote their derivatives in $r$ by primes, e.g., if we set for notational simplicity
$$\theta_0 = 1+r/R_{\nu},$$
then
$$
\theta' = 2\alpha\theta_0^{-1-\delta}, \quad \theta'' = -2(1+\delta)\alpha R_{\nu}^{-1}\theta_0^{-2-\delta}, \quad \ldots.
$$
In particular, since $R_{\nu}^{-1}\theta_0^{-1} \leq r^{-1}$, we have
$$
|\theta^{(k)}| \leq C_{\delta, k}\alpha r^{1-k}\theta_0^{-1-\delta}; \quad k=1, 2, \ldots.
$$

Note that here we can use a slightly simpler exponent $\theta$ than that from \cite{is},
and, accordingly, the proofs get slightly simpler.
This is because our Hamiltonian has a repulsive property due to the potential term $-|x|^{\epsilon}$.

\subsection{A priori super-exponential decay estimates}
In this subsection we prove Proposition~\ref{prop}.
The following commutator estimate plays a major role.

\begin{lem} \label{lem}
Let $\lambda \in {\mathbb R}$, and fix any $\delta \in (0, {\rm min}\{1, \rho\})$ and $\alpha >0$.
Then there exist $c, C >0$ and $n_0 \geq 0$ such that uniformly in $n>m\geq n_0$ and $\nu \geq n_0$, as quadratic form on ${\mathcal D}(H)$,
\begin{equation} \label{keyest1}
\begin{split}
{\rm Im}\bigl[ A\Theta (H-\lambda) \bigr] &\geq cr^{\epsilon/2-1}\theta_0^{-\delta}\Theta - Cr^{\epsilon/2-1}\left(\chi^2_{m-1, m+1}+\chi^2_{n-1, n+1}\right)e^{\theta} \\
  &\quad \ + {\rm Re}\left[ \gamma (H-\lambda) \right],
\end{split}
\end{equation}
where $\gamma = \gamma_{n, m, \nu}$ is a certain function satisfying ${\rm supp}\,\gamma \subseteq {\rm supp} \,\chi_{m, n}$ and $|\gamma| \leq Ce^{\theta}$.
\end{lem}
\begin{proof} 
Let $\lambda \in {\mathbb R}$ and fix any $\delta \in (0, {\rm min}\{1, \rho\})$ and $\alpha>0$.
We choose $n_0 \geq 0$ large enough so that $r(x)=|x|$ on supp\,$\Theta$.
Then we have the following formulae (cf. \eqref{f}).
\begin{equation}
\begin{split}\label{rx}
|dr|^2&=1, \qquad (\nabla^2 f)^{jk}=r^{-\epsilon/2-1}\delta^{jk}-\left( \frac{\epsilon}{2}+1 \right)r^{-\epsilon/2-1}(\nabla r)^j(\nabla r)^k, \\
(\nabla r)^j&=x^jr^{-1}, \qquad \Delta f=(d-\frac{\epsilon}{2}-1)r^{-\epsilon/2-1}, \qquad \Delta r =(d-1)r^{-1}.
\end{split}
\end{equation}
By Lemma~\ref{lem21}, \eqref{equal}, \eqref{rx} and the Cauchy-Schwarz inequality we can estimate
\begin{equation} \label{keyest1-1}
\begin{split}
  &\quad \hspace{-3mm} {\rm Im}\left( A\Theta (H-\lambda) \right) \\
  &\geq \frac{1}{2}p_jr^{-\epsilon/2-1}\Theta\delta^{jk}p_k - \frac{\epsilon+2}{4}p_j(\nabla r)^jr^{-\epsilon/2-1}\Theta(\nabla r)^kp_k + \frac{1}{2}p_j(\nabla r)^jr^{-\epsilon/2}\Theta'(\nabla r)^kp_k \\
  &\quad + \frac{\epsilon}{2}r^{\epsilon/2-1}\Theta - C_1r^{\epsilon/2-1-\rho}\Theta -\frac{4d-3\epsilon-4}{16}r^{-\epsilon/2-1}\theta'^2\Theta - \frac{1}{8}r^{-\epsilon/2}\theta'^3\Theta \\
  &\quad -\frac{3}{8}r^{-\epsilon/2}\theta'\theta''\Theta - \frac{1}{2}{\rm Re}\left( r^{-\epsilon/2}\Theta' (H-\lambda) \right) - C_1Q \\
  &\geq \frac{1}{2}p_jr^{-\epsilon/2-1}\theta_0^{-\delta}\Theta\delta^{jk}p_k - \frac{\epsilon}{4}p_j(\nabla r)^jr^{-\epsilon/2-1}\Theta(\nabla r)^kp_k \\
  &\quad + \frac{1}{2}p_j(\nabla r)^j\left( \theta' - r^{-1}\theta_0^{-\delta} \right)r^{-\epsilon/2}\Theta(\nabla r)^kp_k + \frac{\epsilon}{2}r^{\epsilon/2-1}\Theta - C_1r^{\epsilon/2-1-\rho}\Theta \\
  &\quad -\frac{4d-3\epsilon-4}{16}r^{-\epsilon/2-1}\theta'^2\Theta - \frac{1}{8}r^{-\epsilon/2}\theta'^3\Theta -\frac{3}{8}r^{-\epsilon/2}\theta'\theta''\Theta \\
  &\quad - \frac{1}{2}{\rm Re}\left( r^{-\epsilon/2}\Theta' (H-\lambda) \right) - C_1Q.
\end{split}
\end{equation}
We have introduced for simplicity
\begin{equation} \label{Q1}
\begin{split}
Q &= \left( r^{-\epsilon/2-1-{\rm min}\{1, \rho\}}\chi_{m, n} + |\chi_{m, n}'| + |\chi_{m, n}''| + |\chi_{m, n}'''| \right)e^{\theta} \\
  &\quad \ + p_i\left( r^{-\epsilon/2-1-{\rm min}\{1, \rho\}}\chi_{m, n} + r^{-\epsilon/2}|\chi_{m, n}'| \right)e^{\theta}\delta^{ij}p_j.
\end{split}
\end{equation}
Let us further compute and estimate the terms on the right-hand side of  \eqref{keyest1-1}.
Using a general identity holding for any $g \in C^{\infty}({\mathbb R}^d)$:
\begin{equation} \label{formula}
\frac{1}{2}p_ig\delta^{ij}p_j = \frac{1}{2}{\rm Re}\left( gp_i\delta^{ij}p_j \right) + \frac{1}{4}(\Delta g),
\end{equation}
we estimate the first term of \eqref{keyest1-1} by
\begin{equation} \label{keyest1-2}
\begin{split}
\frac{1}{2}p_jr^{-\epsilon/2-1}\theta_0^{-\delta}\Theta\delta^{jk}p_k
  &\geq {\rm Re}\left( r^{-\epsilon/2-1}\theta_0^{-\delta}\Theta (H-\lambda) \right) + r^{\epsilon/2-1}\theta_0^{-\delta}\Theta \\
  &\quad - C_2r^{\epsilon/2-1-\rho}\Theta + \frac{1}{4}r^{-\epsilon/2-1}\theta_0^{-\delta}\theta'^2\Theta - C_2Q.
\end{split}
\end{equation}
Similarly, we estimate the second term of \eqref{keyest1-1} by
\begin{equation} \label{keyest1-3}
\begin{split}
-\frac{\epsilon}{4}p_j(\nabla r)^jr^{-\epsilon/2-1}\Theta(\nabla r)^kp_k
  &\geq - \frac{\epsilon}{2}{\rm Re}\left( r^{-\epsilon/2-1}\Theta (H-\lambda) \right) - \frac{\epsilon}{2}r^{\epsilon/2-1}\Theta \\
  &\quad - C_3r^{\epsilon/2-1-\rho}\Theta - \frac{\epsilon}{8}r^{-\epsilon/2-1}\theta'^2\Theta - C_3Q.
\end{split}
\end{equation}
We combine the third, seventh and eighth terms of \eqref{keyest1-1} as
\begin{equation} \label{keyest1-4}
\begin{split}
  &\quad \hspace{-3mm} \frac{1}{2}p_j(\nabla r)^j\left( \theta' - r^{-1}\theta_0^{-\delta} \right)r^{-\epsilon/2}\Theta(\nabla r)^kp_k - \frac{1}{8}r^{-\epsilon/2}\theta'^3\Theta -\frac{3}{8}r^{-\epsilon/2}\theta'\theta''\Theta \\
  &= \frac{1}{2}\left( p_j(\nabla r)^j+\frac{i}{2}\theta' \right)\left( \theta'-r^{-1}\theta_0^{-\delta} \right)r^{-\epsilon/2}\Theta\left( (\nabla r)^kp_k-\frac{i}{2}\theta' \right) \\
  &\quad + \frac{i}{4}p_j(\nabla r)^j( \theta'^2-r^{-1}\theta_0^{-\delta}\theta' )r^{-\epsilon/2}\Theta - \frac{i}{4}( \theta'^2-r^{-1}\theta_0^{-\delta}\theta' )r^{-\epsilon/2}\Theta(\nabla r)^kp_k \\
  &\quad - \frac{1}{8}( \theta'^3-r^{-1}\theta_0^{-\delta}\theta'^2 )r^{-\epsilon/2}\Theta - \frac{1}{8}r^{-\epsilon/2}\theta'^3\Theta - \frac{3}{8}r^{-\epsilon/2}\theta'\theta''\Theta \\
  &\geq \frac{1}{2}\left( p_j(\nabla r)^j+\frac{i}{2}\theta' \right)\left( \theta'-r^{-1}\theta_0^{-\delta} \right)r^{-\epsilon/2}\Theta\left( (\nabla r)^kp_k-\frac{i}{2}\theta' \right) \\
  &\quad + \frac{2d-\epsilon-2}{8}r^{-\epsilon/2-1}\theta'^2\Theta - \frac{1}{8}r^{-\epsilon/2-1}\theta_0^{-\delta}\theta'^2\Theta + \frac{1}{8}r^{-\epsilon/2}\theta'\theta''\Theta - C_4Q.
\end{split}
\end{equation}
Substitute \eqref{keyest1-2}, \eqref{keyest1-3} and \eqref{keyest1-4} into \eqref{keyest1-1}, and then it follows that
\begin{equation} \label{keyest1-5}
\begin{split}
  &\quad \hspace{-3mm} {\rm Im}\left( A\Theta (H-\lambda) \right) \\
  &\geq r^{\epsilon/2-1}\theta_0^{-\delta}\Theta - C_5r^{\epsilon/2-1-\rho}\Theta - \frac{\epsilon}{16}r^{-\epsilon/2-1}\theta'^2\Theta + \frac{1}{8}r^{-\epsilon/2}\theta_0^{-\delta}\theta'^2\Theta + \frac{1}{8}r^{-\epsilon/2}\theta'\theta''\Theta \\
  &\quad + \frac{1}{2}\left( p_j(\nabla r)^j+\frac{i}{2}\theta' \right)\left( \theta'-r^{-1}\theta_0^{-\delta} \right)r^{-\epsilon/2}\Theta\left( (\nabla r)^kp_k-\frac{i}{2}\theta' \right) \\
  &\quad + {\rm Re}\left( \left\{r^{-\epsilon/2-1}\theta_0^{-\delta}\Theta -\frac{\epsilon}{2}r^{-\epsilon/2-1}\Theta -\frac{1}{2}r^{-\epsilon/2}\Theta' \right\}(H-\lambda) \right) - C_5Q.
\end{split}
\end{equation}
Using the formula \eqref{formula} we rewrite and bound the remainder operator \eqref{Q1} as
\begin{equation} \label{Q2}
\begin{split}
Q &\leq C_6r^{-\epsilon/2-1-{\rm min}\{1, \rho\}}\Theta + C_6r^{\epsilon/2-1-{\rm min}\{1, \rho\}}\Theta \\
  &\quad \ + C_6r^{\epsilon/2-1}\left(\chi_{m-1, m+1}^2 + \chi_{n-1, n+1}^2\right)e^{\theta} \\
  &\quad \ + 2{\rm Re}\left( \left\{ r^{-\epsilon/2-1-{\rm min}\{1, \rho\}}\Theta + r^{-\epsilon/2}|\chi_{m, n}'|e^{\theta} \right\}(H-\lambda) \right).
\end{split}
\end{equation}
Hence we obtain by \eqref{keyest1-5} and \eqref{Q2}
\begin{equation} \label{uniest}
\begin{split}
  &\quad \hspace{-3mm} {\rm Im}\left( A\Theta (H-\lambda) \right) \\
  &\geq \left( r^{\epsilon/2-1}\theta_0^{-\delta} - C_5r^{\epsilon/2-1-\rho} - C_7r^{-\epsilon/2-1-{\rm min}\{1, \rho\}} + C_7r^{\epsilon/2-1-{\rm min}\{1, \rho\}} \right. \\
  &\quad \left. \ \ + \frac{1}{8}r^{-\epsilon/2}\theta'\theta'' - \frac{\epsilon}{16}r^{-\epsilon/2-1}\theta'^2 + \frac{1}{8}r^{-\epsilon/2}\theta_0^{-\delta}\theta'^2 \right)\Theta \\
  &\quad \ - C_7r^{\epsilon/2-1}\left(\chi_{m-1, m+1}^2 + \chi_{n-1, n+1}^2\right)e^{\theta} + {\rm Re}\left( \gamma(H-\lambda) \right) \\
  &\quad \ + \frac{1}{2}\left( p_j(\nabla r)^j+\frac{i}{2}\theta' \right)\left( \theta'-r^{-1}\theta_0^{-\delta} \right)r^{-\epsilon/2}\Theta\left( (\nabla r)^kp_k-\frac{i}{2}\theta' \right),
\end{split}
\end{equation}
where
$$
\gamma = r^{-\epsilon/2-1}\theta_0^{-\delta}\Theta - \frac{\epsilon}{2}r^{-\epsilon/2-1}\Theta - \frac{1}{2}r^{-1}\Theta' - 2C_5r^{-\epsilon/2-1-{\rm min}\{1, \rho\}}\Theta - 2C_5r^{-\epsilon/2}|\chi_{m, n}'|e^{\theta}.
$$

Now we further restrict parameters.
If we choose sufficiently large $n_0 \geq 0$, the first term is bounded below uniformly in $n>m \geq n_0$ and $\nu \geq 0$ as
\begin{align*}
  &\quad \hspace{-3mm} \left( r^{\epsilon/2-1}\theta_0^{-\delta} - C_5r^{\epsilon/2-1-\rho} - C_7r^{-\epsilon/2-1-{\rm min}\{1, \rho\}} + C_7r^{\epsilon/2-1-{\rm min}\{1, \rho\}} \right. \\
  &\quad \left. \ + \frac{1}{8}r^{-\epsilon/2}\theta'\theta'' - \frac{\epsilon}{16}r^{-\epsilon/2-1}\theta'^2 + \frac{1}{8}r^{-\epsilon/2}\theta_0^{-\delta}\theta'^2 \right)\Theta \\
  &\geq cr^{\epsilon/2-1}\theta_0^{-\delta}\Theta.
\end{align*}
Since
$$\theta' - r^{-1}\theta_0^{-\delta} = \left( 2\alpha \theta_0^{-1} - r^{-1} \right) \theta_0^{-\delta},$$
by retaking $n_0 \geq 0$ larger, if necessary, the fourth term is non-negative for any $n>m\geq n_0$ and $\nu \geq n_0$.
Hence the desired estimate follows.
\end{proof}

\begin{proof}[Proof of Proposition~\ref{prop}]
Let $\lambda \in {\mathbb R}$ and $\phi \in {\mathcal B}_0^*$ be as in the assertion, and fix any $\delta \in (0, {\rm min}\{1, \rho\})$, $\alpha >0$ and $n_0 \geq 0$ in agreement with Lemma~\ref{lem}.
For any function $\phi$ obeying the assumptions of Proposition~\ref{prop} we have $\chi_{m, n}\phi \in H^2_{\rm comp} \subset {\mathcal D}(H)$ for all $n>m \geq 0$.
Note that we may assume $n_0 \geq 3$, so that for all $n > m \geq n_0$
$$
\chi_{m-2, n+2}\phi \in {\mathcal D}(H).
$$
We evaluate the inequality \eqref{keyest1} in the state $\chi_{m-2, n+2}\phi \in {\mathcal D}(H)$, and then obtain for any $n > m \geq n_0$ and $\nu \geq n_0$
\begin{equation} \label{n-infty}
\| ( r^{\epsilon/2-1}\theta_0^{-\delta}\Theta )^{1/2}\phi \|^2 \leq C_m\| \chi_{m-1, m+1}\phi \|^2 + C_{\nu}R_n^{\epsilon/2-1}\| \chi_{n-1, n+1}\phi \|^2.
\end{equation}
The second term on the right-hand side of \eqref{n-infty} vanishes when $n \to \infty$ since $\phi \in {\mathcal B}_0^*$, and consequently by Lebesgue's monotone convergence theorem we have
\begin{equation} \label{nu-infty}
\| (\bar{\chi}_m r^{\epsilon/2-1}\theta_0^{-\delta}e^{\theta})^{1/2}\phi \|^2 \leq C_m\| \chi_{m-1, m+1}\phi \|^2.
\end{equation}
Next we let $\nu \to \infty$ in \eqref{nu-infty} invoking again Lebesgue's monotone convergence theorem, and then it follows that
$$
r^{\epsilon/4-1/2}e^{\alpha r}\phi \in {\mathcal H}.
$$
Consequently this implies $e^{\alpha r}\phi \in {\mathcal B}_0^*$ for any $\alpha \geq 0$.
Hence we are done.
\end{proof}

\subsection{Absence of super-exponentially decaying eigenfunctions}
In this subsection we prove Proposition~\ref{prop2}.

\begin{lem}
Let $\lambda \in {\mathbb R}$ and $\alpha_0 > 0$, and set $\Theta = \chi_{m, n}e^{2\alpha r}$.
Then there exist $c, C > 0$ and $n_0 \geq 0$ such that uniformly in $\alpha > \alpha_0$ and $n > m \geq n_0$, as quadratic forms on ${\mathcal D}(H)$,
\begin{equation} \label{keyest2}
\begin{split}
{\rm Im}\left( A\Theta (H-\lambda) \right)
  &\geq c\alpha^2r^{-\epsilon/2-1}\Theta - C\alpha^2r^{\epsilon/2-1}\left(\chi^2_{m-1, m+1}+\chi^2_{n-1, n+1}\right)e^{2\alpha r} \\
  &\quad + {\rm Re}\left( \gamma (H-\lambda) \right),
\end{split}
\end{equation}
where $\gamma = \gamma_{n, m,}$ is a certain function satisfying ${\rm supp}\,\gamma \subseteq {\rm supp} \,\chi_{m, n}$ and $|\gamma| \leq C\alpha e^{2\alpha r}$.
\end{lem}
\begin{proof}
Fix any $\lambda \in {\mathbb R}$ and $\delta \in (0, {\rm min}\{1, \rho\})$.
Choose $n_0 \geq 0$ large enough.
Then, as with the arguments of the proof of Lemma~\ref{lem}, we can estimate uniformly in $\alpha \geq 0$ and $n > m \geq n_0$ as
\begin{align*}
  &\quad \hspace{-3mm} {\rm Im}\left( A\Theta(H-\lambda) \right) \\
  &\geq \frac{1}{2}p_jr^{-\epsilon/2-1}\Theta\delta^{jk}p_k + \frac{1}{2}p_j(\nabla r)^j\left( 2\alpha-\frac{\epsilon+2}{2}r^{-1} \right)r^{-\epsilon/2}\Theta(\nabla r)^kp_k - C_1r^{\epsilon/2-1-\rho}\Theta \\
  &\quad + \frac{\epsilon}{2}r^{\epsilon/2-1}\Theta - \frac{4d-3\epsilon-4}{4}\alpha^2r^{-\epsilon/2-1}\Theta - \alpha^3r^{-\epsilon/2}\Theta - \frac{1}{2}{\rm Re}\left( r^{-\epsilon/2}\Theta' (H-\lambda) \right) - C_1Q \\
  &\geq \left( \frac{\epsilon}{2}+1 \right)r^{\epsilon/2-1}\Theta - C_2r^{-\epsilon/2-1-\rho}\Theta + \frac{1}{2}\alpha^2r^{-\epsilon/2-1}\Theta \\
  &\quad + \frac{1}{2}\left( p_j(\nabla r)^j+i\alpha \right)\left(2\alpha-\frac{\epsilon+2}{2}r^{-1}\right)r^{-\epsilon/2}\Theta\left( (\nabla r)^kp_k-i\alpha \right) \\
  &\quad + {\rm Re}\left( \left( r^{-\epsilon/2-1}\Theta - \frac{1}{2}r^{-\epsilon/2}\Theta' \right)(H-\lambda) \right) - C_2Q \\
  &\geq \left( \frac{\epsilon}{2}+1 \right)r^{\epsilon/2-1}\Theta - C_3r^{\epsilon/2-1-\rho}\Theta - C_3r^{\epsilon/2-1-{\rm min}\{1, \rho\}}\Theta \\
  &\quad + \frac{1}{2}\alpha^2r^{-\epsilon/2-1}\Theta - C_3(1+\alpha^2)r^{-\epsilon/2-1-{\rm min}\{1, \rho\}}\Theta \\
  &\quad + \frac{1}{2}\left( p_j(\nabla r)^j+i\alpha \right)\left(2\alpha-\frac{\epsilon+2}{2}r^{-1}\right)r^{-\epsilon/2}\Theta\left( (\nabla r)^kp_k-i\alpha \right) \\
  &\quad - C_3(1+\alpha^2)r^{\epsilon/2-1}\left(\chi^2_{m-1, m+1}+\chi^2_{n-1, n+1}\right)e^{2\alpha r} + {\rm Re}\left( \gamma (H-\lambda) \right),
\end{align*}
where
\begin{align*}
Q &= \left( (1+\alpha^2)r^{-\epsilon/2-1-{\rm min}\{1, \rho\}}\chi_{m, n} + (1+\alpha^2)|\chi_{m, n}'| + (1+\alpha)|\chi_{m, n}''| + |\chi_{m, n}'''| \right)e^{2\alpha r} \\
  &\quad \quad +p_i\left( r^{-\epsilon/2-1-{\rm min}\{1, \rho\}}\chi_{m, n} + r^{-\epsilon/2}|\chi_{m, n}'| \right)e^{2\alpha r}\delta^{ij}p_j, \\
\gamma &= r^{-\epsilon/2-1}\Theta - \frac{1}{2}r^{-\epsilon/2}\Theta' - 2C_2r^{-\epsilon/2-1-{\rm min}\{1, \rho\}}\Theta - 2C_2r^{-\epsilon/2}|\chi_{m, n}'|e^{2\alpha r}.
\end{align*}
There we fix any $\alpha_0 > 0$ and choose sufficiently large $n_0 \geq 0$.
Consequently we can easily verify the asserted inequality \eqref{keyest2} uniformly in $\alpha > \alpha_0$ and $n > m \geq n_0$.
Hence we are done.
\end{proof}

\vspace{3mm}
\noindent
{\it Proof of Proposition~\ref{prop2}.}
Let $\lambda \in {\mathbb R}$ and $\phi \in {\mathcal B}_0^*$ be as in the assertion.
Fix any $\alpha_0 > 0$, and choose $n_0 \geq 0$ in agreement with Lemma 3.4.
We may assume that $n_0 \geq 3$, so that for all $n > m \geq n_0$
$$\chi_{m-2, n+2}\phi \in {\mathcal D}(H).$$
Let us evaluate the inequality \eqref{keyest2} in the state $\chi_{m-2, n+2}\phi \in {\mathcal D}(H)$.
Then it follows that for any $\alpha > \alpha_0$ and $n > m \geq n_0$
\begin{equation} \label{n-infty2}
\| r^{-\epsilon/4-1/2}\Theta^{1/2}\phi \|^2 \leq C_m\| \chi_{m-1, m+1}e^{\alpha r}\phi \|^2 + C_1R_n^{\epsilon/2-1}\| \chi_{n-1, n+1}e^{\alpha r} \phi \|^2.
\end{equation}
The second term on the right-hand side of \eqref{n-infty2} vanishes when $n \to \infty$, and hence by Lebesgue's monotone convergence theorem we obtain
$$
\| \bar{\chi}_m^{1/2}r^{-\epsilon/4-1/2}e^{\alpha r}\phi \|^2 \leq C_m\| \chi_{m-1, m+1}e^{\alpha r}\phi \|^2,
$$
or
\begin{equation} \label{nu-infty2}
\| \bar{\chi}_m^{1/2}r^{-\epsilon/4-1/2}e^{\alpha(r-R_{m+2})}\phi \|^2 \leq C_m\| \chi_{m-1, m+1}\phi \|^2.
\end{equation}

Now assume $\bar{\chi}_{m+2}\phi \not\equiv 0$.
The left-hand side of \eqref{nu-infty2} grows exponentially as $\alpha \to \infty$ whereas the right-hand side remains bounded.
This is a contradiction.
Thus $\bar{\chi}_{m+2}\phi \equiv 0$.
By invoking the unique continuation property for the second order elliptic operator $H$ (cf. \cite{wo}) we conclude that $\phi \equiv 0$ globally on ${\mathbb R}^d$.
\hspace{\fill} $\Box$

\end{document}